\documentclass[12pt, reqno]{amsart}
\allowdisplaybreaks[3]
\usepackage[left=3cm,right=3cm,top=2cm,bottom=2cm]{geometry}

\usepackage{amsthm}
\usepackage{amscd}
\usepackage{amsmath}
\usepackage{amsfonts}
\usepackage{amssymb}
\usepackage{graphicx}
\usepackage{amsopn}
\usepackage[usenames]{color}
\usepackage{tikz}
\usetikzlibrary{arrows,shapes,snakes,automata,backgrounds,petri}
\usepackage{textcomp}
\usepackage[all]{xy}

\usepackage[utf8]{inputenc}
\usepackage[T2A]{fontenc}
\usepackage{amsmath,amssymb,amsthm}
\usepackage[english]{babel}

\usepackage{hyperref}

\theoremstyle{theorem}
\newtheorem{theorem}{Theorem}
\newtheorem{corollary}[theorem]{Corollary}
\newtheorem{prop}[theorem]{Proposition}
\newtheorem{lemma}[theorem]{Lemma}

\theoremstyle{definition}

\newtheorem{problem}[theorem]{Problem}

\def\M{\mathcal{M}}
\def\R{\mathbb{R}}
\def\X{\mathcal{X}}
\def\C{\mathcal{C}}
\def\A{\mathcal{A}}
\def\I{\mathcal{I}}

\def\SI{\mathcal{SI}}
\def\K{\mathcal{K}}

\def\T{\mathcal{T}}
\def\h{\mathfrak{h}}

\def\B{\mathcal{B}}
\def\SB{\mathcal{SB}}
\def\S{\Sigma}
\def\Mod{{\rm Mod}}

\def\SMod{{\rm SMod}}
\def\Stab{{\rm Stab}}

\def\id{{\rm id}}
\def\Z{\mathbb{Z}}
\def\Q{\mathbb{Q}}

\def\Homeo{{\rm Homeo}}
\def\SHomeo{{\rm SHomeo}}
\def\H{{\rm H}}
\def\Sp{{\rm Sp}}
\def\SL{{\rm SL}}

\def\TT{{\rm Teich}}

\newcommand{\cd}{\mathop{\mathrm{cd}}}

\newcommand{\sign}{\mathop{\mathrm{sign}}}
\newcommand{\Arf}{\mathop{\mathrm{Arf}}}

\def\M{\mathcal{M}}
\def\R{\mathbb{R}}
\def\X{\mathcal{X}}
\def\C{\mathcal{C}}
\def\A{\mathcal{A}}
\def\I{\mathcal{I}}

\def\SI{\mathcal{SI}}
\def\K{\mathcal{K}}

\def\T{\mathcal{T}}
\def\h{\mathfrak{h}}

\def\B{\mathcal{B}}
\def\SB{\mathcal{SB}}
\def\S{\Sigma}
\def\Mod{{\rm Mod}}

\def\SMod{{\rm SMod}}
\def\Stab{{\rm Stab}}

\def\id{{\rm id}}
\def\Z{\mathbb{Z}}
\def\Q{\mathbb{Q}}

\def\Homeo{{\rm Homeo}}
\def\H{{\rm H}}
\def\Sp{{\rm Sp}}
\def\SL{{\rm SL}}

\def\lk{{\rm lk}}
\def\TT{{\rm Teich}}

\numberwithin{theorem}{section}
\allowdisplaybreaks[3]

\begin{document}

	\title{On the Second Homology of the Genus 3 Hyperelliptic Torelli Group}
	
	\address{National Research University Higher School of Economics, Moscow 119048, Russia}
	\email{spiridonovia@ya.ru}
	\author{Igor Spiridonov}
	

	\maketitle

\begin{abstract}
	Let $s$ be a fixed hyperelliptic involution of the closed, oriented genus $g$ surface $\Sigma_g$. The hyperelliptic Torelli group $\mathcal{SI}_g$ is the subgroup of the mapping class group $\mathrm{Mod}(\Sigma_g)$ consisting of elements that act trivially on $\mathrm{H}_1(\Sigma_g;\mathbb{Z})$ and commute with $s$. It is generated by Dehn twists about $s$-invariant separating curves, and its cohomological dimension is $g-1$. In this paper we study the top homology group $\mathrm{H}_2(\mathcal{SI}_3;\mathbb{Z})$. For each pair of disjoint $s$-invariant separating curves there is a naturally associated abelian cycle in $\mathrm{H}_2(\mathcal{SI}_3;\mathbb{Z})$; we call such cycles \emph{simple}. We show that simple abelian cycles are in bijection with orthogonal (with respect to the intersection form) splittings of $\mathrm{H}_1(\Sigma_3;\mathbb{Z})$ satisfying a simple algebraic condition, and prove that these abelian cycles are linearly independent in $\mathrm{H}_2(\mathcal{SI}_3;\mathbb{Z})$.
\end{abstract}

\section{Introduction}
\subsection{Hyperelliptic Torelli Group}
Let $\S_{g}$ be a compact oriented genus $g$ surface. Let $\Mod(\S_{g})$ be the \textit{mapping class group} of $\S_{g}$, defined by $\Mod(\S_{g}) = \pi_{0}(\Homeo^{+}(\S_{g}))$, where $\Homeo^{+}(\S_{g})$ is the group of orientation-preserving homeomorphisms of $\S_{g}$. The group $\Mod(\S_{g})$ acts on $\H = \H_{1}(\S_{g}, \Z)$ and preserves the algebraic intersection form, so we have a representation $\Mod(\S_{g}) \rightarrow \Sp(2g, \Z)$, which is well-known to be surjective. The kernel $\I_g$ of this representation is known as the \textit{Torelli group}. This can be written as the short exact sequence 
\begin{equation*}
1 \rightarrow \I_g \rightarrow \Mod(\S_{g}) \rightarrow \Sp(2g, \Z) \rightarrow 1.
\end{equation*}
Let $s$ be some fixed hyperelliptic involution on $\S_g$. The \textit{hyperelliptic Torelli group} $\SI_g$ is a subgroup of $\I_g$ consisting of all elements that commute with $s$. Note that any Dehn twist about $s$-invariant separating curve belong to $\SI_g$, therefore this group is contained in the \textit{Johnson kernel} $\K_g$, which is a subgroup of $Mod(\S_{g})$ generated by Dehn twists about separating curves; we have the chain of inclusions
\begin{equation} \label{inclusions}
	\SI_g \subseteq \K_g \subseteq \I_g \subseteq  \Mod(\S_{g}).
\end{equation}
Moreover, Brendle, Margalit and Putman proved \cite{BrendleMP} that the group $\SI_g$ is generated by Dehn twists about $s$-invariant separating curves. 

If $g=1$ we have $\Mod(\S_1) \cong \SL(2, \Z)$, so the groups $\SI_1 = \K_1 = \I_1 = \{\id\}$ are trivial. For $g = 2$, we still have the equality $\SI_2 = \K_2 = \I_2$. McCullough and Miller \cite{McCullough} proved, that the group $\I_2$ is not finitely generated. Later, Mess \cite{Mess} showed that $\I_2$, and hence $\SI_2$, is an infinitely generated free group. For $g \geq 3$, all the inclusions in (\ref{inclusions}) are strict.

\subsection{Results}
Recall that the \textit{cohomological dimension} $\cd(G)$ of a group $G$ is the supremum over all $n$ so that there exists a $G$-module $M$ with $\H^n(G, M) \neq 0$.
In 2013, Brendle, Childers, and Margalit \cite{Brendle} proved that $\cd(\SI_g) = g-1$ and showed that the top homology group $\H_{g-1}(\SI_g, \Z)$ is not finitely generated. In this paper we focus on the structure of the top homology group $\H_{2}(\SI_3, \Z)$.

Recall that for $n$ pairwise commuting elements $h_{1}, \dots, h_{n}$ of the group $G$ one can construct an \textit{abelian cycle} $\A(h_{1}, \dots, h_{n}) \in \H_{n}(G, \Z)$ defined by in the following way. Consider the homomorphism $\phi: \Z^{n} \rightarrow G$ that maps the generator of the $i$-th factor to the $h_{i}$. Then $\A(h_{1}, \dots, h_{n}) = \phi_{*}(\mu_{n})$, where $\mu_{n}$ is the standard generator of $\H_{n}(\Z^{n}, \Z)$.

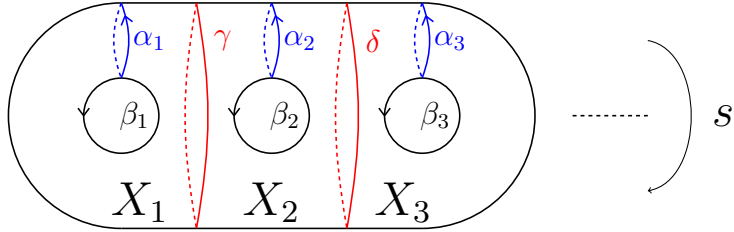
\begin{figure}[h]
	\begin{center}
		\scalebox{0.5}{
			\begin{tikzpicture}
			\draw[red, very thick, dashed] (2,3) to[out = -100, in = 100] (2, -3);
			\draw[red, very thick, dashed] (-2,3) to[out = -100, in = 100] (-2, -3);
			
			\draw[very thick] (-4, 3) to (4, 3);
			\draw[very thick] (-4, -3) to (4, -3);
			\draw[very thick] (-4, 3) arc (90:270:3);
			\draw[very thick] (4, 3) arc (90:-90:3);
			\draw[very thick] (0, 0) circle (1);
			\draw[very thick] (-4, 0) circle (1);
			\draw[very thick] (4, 0) circle (1);
			
			\draw[blue, very thick] (-3.9, 2.65) to (-3.7, 2.5);
			\draw[blue, very thick] (-3.9, 2.65) to (-4, 2.4);
			
			\draw[blue, very thick] (-3.9+4, 2.65) to (-3.7+4, 2.5);
			\draw[blue, very thick] (-3.9+4, 2.65) to (-4+4, 2.4);
			
			\draw[blue, very thick] (-3.9+8, 2.65) to (-3.7+8, 2.5);
			\draw[blue, very thick] (-3.9+8, 2.65) to (-4+8, 2.4);
			
			\draw[very thick] (-5, 0) to (-5.15, 0.2);
			\draw[very thick] (-5, 0) to (-4.85, 0.2);
			
			\draw[very thick] (-5+4, 0) to (-5.15+4, 0.2);
			\draw[very thick] (-5+4, 0) to (-4.85+4, 0.2);
			
			\draw[very thick] (-5+8, 0) to (-5.15+8, 0.2);
			\draw[very thick] (-5+8, 0) to (-4.85+8, 0.2);
			
			\draw[blue, very thick] (4,3) to[out = -70, in = 70] (4, 1);
			\draw[blue, very thick, dashed] (4,3) to[out = -110, in = 110] (4, 1);
			
			\draw[blue, very thick] (0,3) to[out = -70, in = 70] (0, 1);
			\draw[blue, very thick, dashed] (0,3) to[out = -110, in = 110] (0, 1);
			
			\draw[blue, very thick] (-4,3) to[out = -70, in = 70] (-4, 1);
			\draw[blue, very thick, dashed] (-4,3) to[out = -110, in = 110] (-4, 1);
			
			\draw[red, very thick] (2,3) to[out = -80, in = 80] (2, -3);
			\draw[red, very thick] (-2,3) to[out = -80, in = 80] (-2, -3);
			
			\node[red, scale = 2] at (2.7, 2) {$\delta$};
			\node[red, scale = 2] at (2.7-4, 2) {$\gamma$};
			
			\node[blue, scale = 2] at (4.75, 2) {$\alpha_3$};
			\node[blue, scale = 2] at (4.75-4, 2) {$\alpha_2$};
			\node[blue, scale = 2] at (4.75-8, 2) {$\alpha_1$};
			
			\node[scale = 2] at (4.35, 0) {$\beta_3$};
			\node[scale = 2] at (4.35-4, 0) {$\beta_2$};
			\node[scale = 2] at (4.35-8, 0) {$\beta_1$};
			
			\node[scale = 3] at (3.5, -2.3) {$X_3$};
			\node[scale = 3] at (-3.5, -2.3) {$X_1$};
			\node[scale = 3] at (0, -2.3) {$X_2$};
			
			\draw[very thick, dashed] (8, 0) to (10, 0);
			\draw[->, thick] (10, 2) to[out = -10, in = 10] (10, -2);
			\node[scale = 3] at (12, 0) {$s$};

			\end{tikzpicture}} \end{center}
	\caption{The surface $\S_3$, the involution $s$, and the curves $\alpha_i$ and $\beta_i$.}
	\label{S}
\end{figure}

We denote by $T_{\gamma}$ the left Dehn twist about a curve $\gamma$.
We say that $(\gamma, \delta)$ is a \textit{separating pair}, if $\gamma$ and $\delta$ are disjoint nonisotopic separating cures on $\S_3$. We say that a separating pair $(\gamma, \delta)$ is \textit{symmetric}, if the curves $\gamma$ and $\delta$ are $s$-invariant. 
An example of such pair is shown in Fig \ref{S}. If $(\gamma, \delta)$ is a symmetric separating pair, then the Dehn twists $T_\gamma$ and $T_\delta$ belong to $\SI_3$ and commute with each other. Hence one can consider the correspondent abelian cycle $\A(T_\gamma, T_\delta) \in \H_2(\SI_3, \Z)$. Such abelian cycles will be called \textit{simple}. 
One can deduce from the results of Brendle and Farb \cite{BrendleFarb} or Gaifullin \cite{Gaifullin_J} that each simple abelian cycle is nonzero and has the infinite order, see Proposition \ref{nonzero}. The main goal of the present paper is to study the relations between simple abelian cycles in the group $\H_2(\SI_3, \Z)$.

First, we should describe explicitly the set of all simple abelian cycles.
Let $V_1, V_2, V_3 \subset \H = \H_1(\S_3, \Z)$ be subgroups of rank 2. We say that $(V_1, V_2, V_3)$ is an \textit{orthogonal splitting of} $\H$, if $\H = V_1 \oplus V_2 \oplus V_3$ and the subgroups $V_1, V_2, V_3$ are pairwise orthogonal w.r.t the intersection form. Let $(\gamma, \delta)$ be a separating pair. Denote by $X_1$, $X_2$, $X_3$ the connected components of $\S_3 \setminus \{\gamma, \delta\}$ as shown in Fig. \ref{S}. Denote $\H_{X_i} = \H_1(X_i, \Z) \subset \H$ for $i = 1, 2, 3$. Then $(\H_{X_1}, \H_{X_2}, \H_{X_3})$ is an orthogonal splitting of $\H$; we say that this splitting \textit{corresponds} to the separating pair $(\gamma, \delta)$. We say that an orthogonal splitting is \textit{symmetric}, if it corresponds to some symmetric separating pair $(\gamma, \delta)$. In Section \ref{Sec3}, we prove the following result.

\begin{prop} \label{inj}
	Let $(\gamma, \delta)$ and $(\gamma', \delta')$ be two symmetric separating pairs, such that the corresponding orthogonal splittings of $\H$ coincide. Then $(\gamma, \delta)$ and $(\gamma', \delta')$ are $\SI_3$-equivalent.
\end{prop}

Let $(V_1, V_2, V_3)$ be a symmetric orthogonal splitting of $\H$. Let $(\gamma, \delta)$ be a symmetric separating pair, such that $(V_1, V_2, V_3)$ corresponds to $(\gamma, \delta)$. Proposition \ref{inj} implies that all such symmetric separating pairs are $\SI_3$-equivalent, so the simple abelian cycle $\A_{V_1, V_2, V_3} = \A(T_\gamma, T_\delta) \in \H_2(\SI_3, \Z)$ does not depend on the choice of the curves $\gamma$ and $\delta$, and hence is well-defined. Obviously, any simple abelian cycle has such form for some symmetric orthogonal splitting $(V_1, V_2, V_3)$. We also have the relations $\A_{V_1, V_2, V_3} = - \A_{V_3, V_2, V_1}$.

The next question is which orthogonal splittings of $\H$ are symmetric. In order to answer it, we need to recall the definition of the Arf-invariant of an Sp-quadratic form.
Consider the group $W \cong \Z^{2n}$ with a non-degenerate symplectic form $J(x, y) = x \cdot y$ with determinant $\pm1$.
Recall that an \textit{Sp-quadratic form} on $W$ is a map $\omega: W \to \Z / 2\Z$ satisfying 
$$\omega(x\pm y) = \omega(x) + \omega(y) + (x \cdot y \mod 2).$$
The \textit{Arf-invariant} of $\omega$ is 
$$\Arf(\omega) = \sum_{i=1}^n \omega(x_i)\omega(y_i) \in \Z / 2\Z,$$
where $\{x_1, y_1, \dots, x_n, y_n\}$ is a symplectic basis of $W$. The Arf-invariant is well-defined, i.e. does not depend of the choice of the symplectic basis. Note that classically Sp-quadratic forms are defined on vector spaces over $\Z / 2\Z$. Following this approach, we say that an Sp-quadratic form on $W$ is the same as an Sp-quadratic form on $W$ reduced modulo $2$.

Let us identify $\S_3$ and $s$ with the surface and the involution shown in Fig. \ref{S}. We also fix the curves $\alpha_i$ and $\beta_i$ for $i = 1, 2, 3$, shown in Fig \ref{S}. The homology classes $\{a_1, b_1, a_2, b_2, a_3, b_3\}$, where $a_i = [\alpha_i]$ and $b_i = [\beta_i]$, form a symplectic basis of $\H$.
Let us fix the Sp-quadratic form $\omega_0$ on $\H$ defined by
\begin{equation*}
\omega_0(a_1) = \omega_0(a_3) = \omega_0(b_1) = \omega_0(b_2) = \omega_0(b_3) = 1, \; \; \; \; \omega_0(a_2) = 0.
\end{equation*}
We will prove the following result.

\begin{prop} \label{image}	
	 Let $(V_1, V_2, V_3)$ be an orthogonal splitting of $\H$. Then this splitting is symmetric if and only if the following condition hold.
	 \begin{equation} \label{cond}
	 \Arf(\omega_0|_{V_1}) = \Arf(\omega_0|_{V_3}) = 1, \; \; \; \;\Arf(\omega_0|_{V_2}) = 0.
	 \end{equation}
\end{prop}

Propositions \ref{inj} and \ref{image} imply that any simple abelian cycle has a form $\A_{V_1, V_2, V_3}$, where $(V_1, V_2, V_3)$ is an orthogonal splitting of $\H$, satisfying the condition \ref{cond}.
The main result of the present paper is as follows.

\begin{theorem} \label{mainth}
The simple abelian cycles $\A_{V_1, V_2, V_3} \in \H_2(\SI_3, \Z)$, where $(V_1, V_2, V_3)$ runs over the set of all orthogonal splittings of $\H$ satisfying (\ref{cond}), are linearly independent over $\Z$ after taking the quotient by the relations $\A_{V_1, V_2, V_3} = -\A_{V_3, V_1, V_2}$.
\end{theorem}

Denote by $\A_3 \subseteq \H_2(\SI_3, \Z)$ the subgroup generated by all simple abelian cycles. The following question is very natural.

\begin{problem}
	Is the inclusion $\A_3 \subseteq \H_2(\SI_3, \Z)$ strict? In other words, does there exist a homology class in $\H_2(\SI_3, \Z)$, which is not a linear combination of simple abelian cycles?
\end{problem}

\subsection{History} The groups $\Mod(\S_g)$, $\I_g$ and $\SI_g$ also arise in algebraic geometry in the following way. Assume that $g \geq 2$. Recall that the \textit{Teichmuller space} $\TT_g$ is the moduli space of complex structures (or, equivalently, hyperbolic metrics) on $\S_g$. This space is contractible, moreover, we have a homeomorphism $\TT_g \cong \R^{6g-6}$. The natural action of $\Mod(\S_g)$ on $\TT_g$ has finite stabilizers, and the quotient space $\M_g = \TT_g / \Mod(\S_g)$ is the \textit{moduli space} of smooth complex curves. From the topological point of view, $\M_g$ is a complex orbifold.
The action of $\I_g$ on the Teichmuller space is free, so the quotient $\T_g = \TT_g / \I_g$ is a maifold; it called the \textit{Torelli space}. It is the moduli space of smooth complex curves with homology framing, i.e. with a fixed symplectic basis in $\H$. The Torelli space is an Eilenberg-MacLane space for the Torelli group: we have $\H^*(\I_g, \Z) \cong \H^*(\T_g, \Z)$. These cohomology groups can also be interpreted as the set of characteristic classes of homologically trivial surface bundles. Denote $\h_g$ the Siegel upper half space.
Recall that there is the well-defined \textit{period map} $\T_g \to \h_g$, which is 2-fold branched cover onto its image in the case $g \geq 3$ (if $g=2$, then the period map is an embedding). For $g \geq 3$, the branch locus of the period map is a subspace of $\T_g$ consisting of all points, projecting to the hyperelliptic locus via the natural map $\T_g \to \M_g$. The branch locus is not connected. However, its components are pairwise diffeomorphic. It turns out that each of these components is an Eilenberg-MacLane space for the hyperelliptic Torelli group \cite{Kordek}.

A natural problem is to study the homology of $\I_g$, $\K_g$ and $\SI_g$. The first homology of the Torelli group was computed explicitly by Johnson \cite{Johnson3}. Bestvina, Bux and Margalit \cite{Bestvina} in 2007 constructed the contractible complex of cycles $\B_{g}$, on which the Torelli group $\I_g$ acts cellularly. Using the spectral sequence associated with this action, they showed, that the group $\I_g$ has cohomological dimension $3g-5$ and that the top homology group $\H_{3g-5}(\I_g, \Z)$ is not finitely generated. Gaifullin \cite{Gaifullin_T} in 2019 proved that for $2g-3 \leq k \leq 3g-5$ the homology group $\H_{k}(\I_g, \Z)$ contains a free abelian subgroup of an infinite rank. In 2023, Minahan \cite{Minahan} proved that for $g \geq 51$, the vector space $\H_2(\I_g; \Q)$ has finite dimension. In 2025, Minahan and Putman \cite{Minahan} extended this result for $g \geq 5$ and explicitly computed the group $\H_2(\I_g; \Q)$ for $g \geq 6$.

There are also some results on the homology of the Torelli group in the case $g=3$. Applying stratified Morse theory to the image of the period map $\T_3 \to \h_3$, Hain \cite{Hain} computed the invariant part of $\H_*(\I_3, \Z[1/2])$ w.r.t. the action of the  hyperelliptic involution. 
Gaifullin in \cite{Gaifullin_T3} obtained a partial result towards the conjecture that $\H_2(\I_3, \Z)$ is not finitely generated.
The author \cite{Spiridonov_T3} computed explicitly the group $\H_4(\I_3, \Z)$.

For the Johnson kernel, Bestvina, Bux, and Margalit~\cite{Bestvina} showed that the cohomological dimension of $\K_g$ is equal to $2g-3$, and Gaifullin \cite{Gaifullin_J} proved that the top homology group $\H_{2g-3}(\K_g,\Z)$ is not finitely generated. The author \cite{Spiridonov_Kg} computed explicitly the subgroup of $\H_{2g-3}(\K_g, \Z)$ generated by simple abelian cycles. Concerning the first homology, initially Dimca and Papadima  \cite{DiPa13} showed that $\H_1(\K_g, \Q)$ is finite-dimensional for $g \geq 4$. Later, from the works of Ershov and He \cite{ErHe18}, and Church, Ershov and Putman \cite{CEP22}, it follows that for $g \geq 4$ the groups $\K_g$ and $\H_1(\K_g, \Z)$ are finitely generated. Finally, Gaifullin \cite{Gaifullin_K3} proved that $\H_1(\K_3, \Z)$ is also finitely generated.

Much less in known about $\H_*(\SI_g, \Z)$. In order to study these homology groups, Brendle, Childers, and Margalit \cite{Brendle} in 2013 introduced another CW-complex $\SB_g$ called the \textit{complex of symmetric cycles}, on which $\SI_g$ acts cellularly. Using the similar ideas to \cite{Bestvina}, they proved that $\cd(\SI_g) = g-1$ and showed that the top homology group $\H_{g-1}(\SI_g, \Z)$ is not finitely generated.

\subsection{Structure of the paper} In Section \ref{Sec2}, we recall some  results on the symmetric mapping class groups and Birman-Craggs homomorphisms. Section \ref{Sec3} contain proofs of Propositions \ref{inj} and \ref{image}. In Section \ref{Sec4}, we prove Theorem \ref{mainth}; the main approach is the spectral sequence associated with the action of $\SI_3$ on $\SB_3$.

\subsection{Acknowledgements}
The author would like to thank his scientific advisor A.~Gaifullin for useful discussions and constant attention to this work.

\section{Preliminaries on symmetric mapping class groups}\label{Sec2}

\subsection{Birman-Hilden theorem}
Let $\SHomeo^+(\S_g)$ be the group of orientation-preserving homeomorphisms of $\S_g$ commuting with the hyperelliptic involution $s$. Note that $s$ has $2g+2$ fixed points on $\S_g$, so the quotient space $\S_g / s$ is a sphere with $2g+2$ cone points. Since the set of cone points is invariant w.r.t. every element of $\SHomeo^+(\S_g)$, we have a short exact sequence
\begin{equation*} \label{Homeoex}
1 \rightarrow \left\langle s \right\rangle  \rightarrow \SHomeo^+(\S_g) \rightarrow \Homeo^+(\S_{0, 2g+2}) \rightarrow 1.
\end{equation*}
Let us define the \textit{symmetric mapping class group} $\SMod(\S_g)$ as the group of isotopy classes of $\SHomeo^+(\S_g)$ (we do not require isotopies to be $s$-invariant). The Birman-Hilden Theorem \cite[Theorem 7]{BH11} says that there is a short exact sequence
\begin{equation} \label{Smodex}
1 \rightarrow \left\langle s \right\rangle  \rightarrow \SMod(\S_g) \rightarrow \Mod(\S_{0, 2g+2}) \rightarrow 1,
\end{equation}
where by $\Mod(\S_{0, n})$ we denote the mapping class group of the sphere with $n$ punctures.
In particular, this implies that if two elements of $\SHomeo^+(\S_g)$ are isotopic, then they are isotopic via an $s$-invariant isotopy. Obviously we have $\SI_g = \I_g \cap \SMod(\S_g)$.

We also need to describe the image of $\SMod(\S_g)$ under the symplectic representation $\rho: \Mod(\S_g) \to \Sp(2g, \Z)$. Consider the symmetric group $S_{2g+2}$ which can be realised as the permutation group of $2g + 2$ punctures on $\S_{0, 2g+2}$. Then we have the following result; the author is grateful to Will Sawin \cite{SawinMO} for pointing it out.

\begin{prop}\cite[Theorem 1]{ACampo} \label{fact_2}
	Let $g \geq 3$. Then we have
	$$[\Sp(2g, \Z) : \rho(\SMod(\S_g))] = \dfrac{|\Sp(2g, \Z/2\Z)|}{|S_{2g+2}|} = \dfrac{2^{g^2}\prod_{i=1}^{g}\bigl(2^{2i}-1\bigr)}{(2g+2)!}.$$
\end{prop}

\begin{corollary}\label{fact_2_cor}
	Let $g \geq 3$. Then the number of distinct $\I_g$-orbits of hyperelliptic involutions on $\S_g$ equals to 
	$$\dfrac{2^{g^2}\prod_{i=1}^{g}\bigl(2^{2i}-1\bigr)}{(2g+2)!}.$$
	In particular, for the case $g=3$, we obtain exactly $36$ distinct $\I_3$-orbits.
\end{corollary}

\begin{proof}
	Fix a hyperelliptic involution $s\in \Mod(\S_g)$. Every hyperelliptic involution on $\S_g$ is conjugate to $s$, hence the set $\mathcal H_g$ of hyperelliptic involutions is a single $\Mod(\S_g)$--conjugacy class. The stabiliser of $s$ under conjugation is its centraliser, hence
	$$\Stab_{\Mod(\S_g)}(s)=\SMod(\S_g),$$
	so we have a $\Mod(\S_g)$--equivariant identification
	$$\mathcal H_g \cong \Mod(\S_g)/\SMod(\S_g).$$
	The action of $\I_g$ on $\mathcal H_g$ corresponds to the left action of $\I_g$ on $\Mod(\S_g)/\SMod(\S_g)$, so the set of $\I_g$--orbits is naturally identified with the set of double cosets
	$$\I_g \backslash \Mod(\S_g) / \SMod(\S_g).$$
	Since $\I_g \trianglelefteq \Mod(\S_g)$, this set is in bijection with the coset space $\Mod(\S_g)/(\I_g \SMod(\S_g))$, and therefore the number of $\I_g$--orbits equals
	$$[\Mod(\S_g):\I_g \SMod(\S_g)].$$
	Using $\ker(\rho)=\I_g$, we have $\Mod(\S_g)/\I_g \cong \Sp(2g,\Z)$, and the image of $\I_g \SMod(\S_g)$ in this quotient equals $\rho(\SMod(\S_g))$. Hence
	$$[\Mod(\S_g):\I_g \SMod(\S_g)] = [\Sp(2g,\Z):\rho(\SMod(\S_g))].$$
	The claim now follows from Proposition~\ref{fact_2}.
\end{proof}

\subsection{Birman-Craggs homomorphisms} Given an Sp-quadratic form $\omega$ on $\H_1(\S_g, \Z)$ with $\Arf(\omega) = 0$, Birman and Craggs \cite{BC} constructed a homomorphism $\rho_\omega: \I_g \to \Z / 2\Z$ called \textit{Birman-Craggs homomorphism}. 
Recall that the \textit{extended Torelli group} $\widehat{\I}_g$ is the extension of $\I_g$ by any hyperelliptic involution. Gaifullin \cite[Theorem 1]{Gaifullin_BCJ} proved that a Birman-Craggs homomorphism has a natural extension on the group $\widehat{\I}_g$. Moreover, he showed that its value on a hyperelliptic involution can computed explicitly via simple formula. We give this formula in the case $g=3$ and the involution $s$ shown in Fig. \ref{S}:
\begin{equation*}
\rho_\omega(s) = \sum_{1 \leq i < j \leq 3} \omega(a_i)\omega(b_i)(\omega(a_i)+1)\omega(b_i).
\end{equation*}
Recall that $\Arf(\omega) = 0$, so
\begin{equation*}
\rho_\omega(s) =
\end{equation*}
\begin{equation*}
= \omega(a_1)\omega(b_1)(\omega(a_2)+1)\omega(b_2) + \omega(a_1)\omega(b_1)(\omega(a_3)+1)\omega(b_3) + \omega(a_2)\omega(b_2)(\omega(a_3)+1)\omega(b_3) =
\end{equation*}
\begin{equation*}
= \omega(a_1)\omega(b_1)(\omega(a_2)+1)\omega(b_2) + \Bigl( \omega(a_1)\omega(b_1) + \omega(a_2)\omega(b_2) \Bigr) (\omega(a_3)+1)\omega(b_3) = 
\end{equation*}
\begin{equation*}
= \omega(a_1)\omega(b_1)(\omega(a_2)+1)\omega(b_2) + \Bigl( \Arf(\omega)+ \omega(a_3)\omega(b_3) \Bigr) (\omega(a_3)+1)\omega(b_3) = 
\end{equation*}
\begin{equation*}
= \omega(a_1)\omega(b_1)(\omega(a_2)+1)\omega(b_2) + \omega(a_3)\omega(b_3) (\omega(a_3)+1)\omega(b_3) = 
\end{equation*}
\begin{equation*}
= \omega(a_1)\omega(b_1)(\omega(a_2)+1)\omega(b_2).
\end{equation*}
Hence we obtain
\begin{equation} \label{BCHI}
\rho_\omega(s) = \omega(a_1)\omega(b_1)(\omega(a_2)+1)\omega(b_2).
\end{equation}
Formula (\ref{BCHI}) immediately implies the following result.
\begin{prop} \label{fact_1}
	For any hyperelliptic involution $s'$ on $\S_3$ there is a unique Sp-quadratic form $\omega'$ on $\H_1(\S_3, \Z)$ with $\Arf(\omega') = 0$ satisfying $\rho_{\omega'}(s') = 1$.
\end{prop}
For example, for the involution $s$ shown in Fig. \ref{S}  this Sp-quadratic form from Proposition \ref{fact_1} is exactly $\omega_0$.

\begin{corollary} \label{fact}
	The $\I_3$-orbit of a hyperelliptic involution $s'$ is uniquely determined by a unique Sp-quadratic form $\omega'$ on $\H_1(\S_3, \Z)$ with $\Arf(\omega') = 0$ satisfying $\rho_{\omega'}(s') = 1$.
\end{corollary}

\begin{proof}
	Straightforward computation shows that there are $36$ distinct Sp-quadratic form $\omega'$ on $\H_1(\S_3, \Z)$ with $\Arf(\omega') = 0$. Since by Corollary \ref{fact_2_cor} there exist exactly $36$ distinct $\I_3$-orbits of hyperelliptic involutions on $\S_3$, then the result follows from Proposition \ref{fact_1}.
\end{proof}

\section{Classification of the simple abelian cycles}\label{Sec3}
The main goal of this section is to prove Propositions \ref{inj} and \ref{image}. First we need to prove some auxiliary results. From now on and throughout the paper, we denote $\H = \H_1(\S_3, \Z)$.

\subsection{A refined Euclidean algorithm}
Consider the group $\Z^2 = \left\langle a, b \right\rangle$ with a symplectic form given by the matrix
\begin{equation*}\label{int}
\begin{pmatrix}
0 & 1\\
-1 & 0
\end{pmatrix}.
\end{equation*}
Consider the standard generators of $\SL(2, \Z)$
\begin{equation*} \label{mat}
R_1 = \begin{pmatrix}
1 & 1\\
0 & 1
\end{pmatrix}, \;\;\; R_2 = \begin{pmatrix}
1 & 0\\
1 & 1
\end{pmatrix}.
\end{equation*}
The following lemma in a simple corollary of the Euclidean algorithm.
\begin{lemma} \label{Euc1}
	Let $(a', b')$ be any symplectic basis of $\Z^2$. Then $(a', b')$ is $\SL(2, \Z)$-equivalent to $(a, b)$. 
\end{lemma}
\begin{proof}
	Consider the matrix
	\begin{equation*}
	X=\begin{pmatrix}
	x_1 & x_2\\
	y_1 & y_2
	\end{pmatrix},
	\end{equation*}
	where $a' = x_1 a+ x_2 b$ and $b' = y_1 a+ y_2 b$. Since $(a', b')$ is a symplectic basis, it follows that $\det(X) = 1$. Hence $\gcd(x_1, y_1) = 1$.
	
	We need to transform $X$ to the identity matrix $E$ acting from the left by matrices $R_1, R_2 \in \SL(2, \Z)$, i.e. applying the standard Euclidean algorithm to the rows of $X$. 
	Since $\gcd(x_1, y_1) = 1$, then there exists $D \in \SL(2, \Z)$ such that $D(x_1, y_1)^T = (1, 0)^T$. Since  $\det(X) = 1$, then
	\begin{equation*}
	DX = \begin{pmatrix}
	1 & z\\
	0 & 1
	\end{pmatrix},
    \end{equation*}
    Finally, we have $R_2^{-z}DX = E$. This proves the lemma.
\end{proof}

Define an Sp-quadratic form $\nu$ on $\Z^2$ by $\nu(a) = 0$ and $\nu(b) = 1$. Consider the subgroup $G = \left\langle R^2_1, R_2\right\rangle  \subset \SL(2, \Z)$.

We prove the following modification of Lemma \ref{Euc1}.
\begin{lemma} \label{Euc2}
	Let $(a', b')$ be a symplectic basis of $\Z^2$ satisfying $\nu(a') = 0$ and $\nu(b') = 1$. Then $(a', b')$ is $G$-equivalent to $(a, b)$.
\end{lemma}
\begin{proof}
	As in the proof of Lemma \ref{Euc1}, consider the matrix
	\begin{equation*}
	X=\begin{pmatrix}
	x_1 & x_2\\
	y_1 & y_2
	\end{pmatrix},
	\end{equation*}
	where $a' = x_1 a+ x_2 b$ and $b' = y_1 a+ y_2 b$. Since $(a', b')$ is a symplectic basis, it follows that $\det(X) = 1$, so $\gcd(x_1, y_1) = 1$. We have 
	$$1 = \nu(b') = \nu(y_1 a+ y_2 b) = (y_2 + y_1 y_2 \mod 2) = ((y_1 + 1)y_2 \mod 2).$$
	Therefore $y_1$ is even and $y_2$ is odd.
	
	We need to transform $X$ to the identity matrix $E$ acting from the left by matrices $R^2_1, R_2 \in \SL(2, \Z)$. We apply the following ``refined Euclidean algorithm'' consisting of iterations of the following step.
	
	$\bullet$ if $y_1 = 0$, then stop,
	
	$\bullet$ if $y_1 \neq 0$ and $|x_1|>|y_1|$, then apply $R_2^{-\sign (x_1y_1)}$,
	
	$\bullet$ if $y_1 \neq 0$ and $|x_1|<|y_1|$, then apply $R_1^{-2\sign (x_1y_1)}$.
	
	First, we claim that this algorithm is well-defined. Indeed, at each iteration we have that $y_1$ is even and $y_2$ is odd, i.e. $\nu(y_1 a+ y_2 b) = 1$. Hence the situation $|x_1| = |y_1|$ is not possible because $y_1$ is even and $\det (X) = 1$. Note that $x_1 = 0$ is also not possible. 
	
	We also claim that it stops in finitely many iterations. This follows from the fact that the number $|x_1|+|y_1|$ strictly decreases at each step, because we have $x_1 \neq 0$, $y_1 \neq 0$ and $|x_1| \neq |y_1|$.
	
	After this algorithm stops, we obtain a matrix
	$$X' = \begin{pmatrix}
	1 & z\\
	0 & 1
	\end{pmatrix} \;\; \mbox{or} \;\;\; X'' = \begin{pmatrix}
	-1 & z\\
	0 & -1
	\end{pmatrix}.$$
	In the first case we have $R_2^{-z}X' = E$. In the second case we have 
	$$R_1^{-2} X'' = \begin{pmatrix}
	-1 & z\\
	2 & -1-2z
	\end{pmatrix}, \; R_2 R_1^{-2} X'' = \begin{pmatrix}
	1 & -1-z\\
	2 & -1-2z
	\end{pmatrix},$$
	$$R_1^{-2}R_2 R_1^{-2} X'' = \begin{pmatrix}
	1 & -1-z\\
	0 & 1
	\end{pmatrix}.$$
	Hence $R_2^{z+1}R_1^{-2}R_2 R_1^{-2} X'' = E$. 
	This proves the lemma.
\end{proof}

\subsection{Proof of Proposition \ref{inj}}

First we prove the following statement.
\begin{lemma}\label{orbit1}
	Let $(\gamma, \delta)$ and $(\gamma', \delta')$ be two symmetric separating pairs. Then $(\gamma, \delta)$ and $(\gamma', \delta')$ are $\SMod(\S_3)$-equivalent.
\end{lemma}
\begin{proof}
	Denote by $\widetilde{\gamma}$, $\widetilde{\gamma}'$, $\widetilde{\delta}$ and $\widetilde{\delta}'$ the images of the curves $\gamma$, $\gamma'$, $\delta$ and $\delta'$ under the projection $\S_3 \to \S_3 / s$. We can assume that these curves do not contain the cone points, so we can consider them as the curves on $\S_{0, 8}$. Consider the closures of the connected components of $\S_{0, 8} \setminus \{\widetilde{\gamma}, \widetilde{\delta}\}$. Two of them are spheres with $3$ punctures and one boundary component. The third one is a sphere with $2$ punctures and $2$ boundary components. The same is true for the closures of the connected components of $\S_{0, 8} \setminus \{\widetilde{\gamma}', \widetilde{\delta}'\}$. Hence the pairs $(\widetilde{\gamma}, \widetilde{\delta})$ and $(\widetilde{\gamma}', \widetilde{\delta}')$ are $\Mod(\S_{0, 8})$-equivalent. The exactness of \ref{Smodex} implies that $(\gamma, \delta)$ and $(\gamma', \delta')$ are $\SMod(\S_3)$-equivalent.
\end{proof}

\begin{proof}[Proof of Proposition \ref{inj}.]
	Lemma \ref{orbit1} implies that there exists $\phi \in \SMod(\S_3)$ such that $\phi(\gamma') = \gamma$ and $\phi(\delta') = \delta$. Lemma \ref{orbit1} also implies that without loss of generality we can assume that the curves $\gamma$ and $\delta$ are shown in Fig. \ref{S}. Denote by $X_1$, $X_2$ and $X_3$ the connected components of $\S_3 \setminus \{\gamma, \delta\}$ as shown in Fig. \ref{S}. We have $V_i = \H_1(X_i, \Z)$. The subgroups $V_i \subset \H$ are $\phi$-invariant.
	
	Consider the subgroups 
	\begin{gather*}
	K_1 = \left\langle T_{\alpha_1}, T_{\beta_1} \right\rangle  \subset \SMod(\S_3),\\
	K_2 = \left\langle T_{\alpha_2}T_{s(\alpha_2)}, T_{\beta_2} \right\rangle  \subset \SMod(\S_3),\\
	K_3 = \left\langle T_{\alpha_3}, T_{\beta_3} \right\rangle  \subset \SMod(\S_3).
	\end{gather*}
	For each $i = 1, 2, 3$ the group $K_i$ acts on $V_i \cong \Z^2 = \left\langle a_i, b_i \right\rangle $, so we have the homomorphisms $\tau_i: K_i \to \SL(2, \Z)$. We have
	\begin{gather*}
	\tau_1(T_{\alpha_1}) = R_1, \;\;\; \tau_1(T_{\beta_1}) = R_2,\\
	\tau_2(T_{\alpha_2}T_{s(\alpha_2)}) = R^2_1, \;\;\; \tau_2(T_{\beta_2}) = R_2,\\
	\tau_3(T_{\alpha_3}) = R_1, \;\;\; \tau_3(T_{\beta_3}) = R_2.
	\end{gather*}
	We obtain $\tau_1(K_1) = \tau_3(K_3) = \SL(2, \Z)$ and $\tau_2(K_2) = G$, where $G = \left\langle R^2_1, R_2\right\rangle  \subset \SL(2, \Z)$.
	
	Note that $([\phi(\alpha_1)], [\phi(\beta_1)])$ is a symplectic basis of $V_1$. Hence Lemma \ref{Euc1} implies that there exists $A_1 \in \SL(2, \Z)$, such that $A_1 [\phi(\alpha_1)] = a_1$ and $A_1 [\phi(\beta_1)] = b_1$. Choose $\psi_1 \in K_1$ such that $\tau_1(\psi_1) = A_1$.
	
	Similarly, there exists $A_3 \in \SL(2, \Z)$, such that $A_3 [\phi(\alpha_3)] = a_3$ and $A_3 [\phi(\beta_3)] = b_3$. Choose $\psi_3 \in K_3$ such that $\tau_3(\psi_3) = A_3$.

	Finally, $([\phi(\alpha_2)], [\phi(\beta_2)])$ is a symplectic basis of $V_2$. The Sp-quadratic form $\omega_0|_{V_2}$ has is defined by $\omega_0(a_2) = 0$ and $\omega_0(b_2) = 1$. Since $\phi \in \SMod(\S_3)$ it follows that  $\omega_0([\phi(\alpha_2)])=0$ and $\omega_0([\phi(\beta_1)]) = 0$. Hence Lemma \ref{Euc2} implies that there exists $A_2 \in G \subset \SL(2, \Z)$, such that $A_2 [\phi(\alpha_2)] = a_2$ and $A_2 [\phi(\beta_2)] = b_2$. Choose $\psi_2 \in K_2$ such that $\tau_2(\psi_2) = A_2$.
	
	Consider the element $\Phi = \psi_1 \circ \psi_2 \circ \psi_3 \circ \phi \in \SMod(\S_3)$. We have $\Phi(a_i) = a_i$ and $\Phi(b_i) = b_i$ for $i = 1,2,3$. Hence $\Phi \in \SI_3$. Moreover, we have $\Phi(\gamma') = \gamma$ and $\Phi(\delta') = \delta$. This implies the proposition.
\end{proof}

\subsection{Proof of Proposition \ref{image}}

\begin{proof}[Proof of Proposition \ref{image}]
	The implication from left to right is straightforward. Indeed, if $(V_1, V_2, V_3)$ is an orthogonal symmetric splitting, correspondent to symmetric separating pair $(\gamma, \delta)$, then we can chose a symplectic basis $\{a_1, b_1, a_2, b_2, a_3, b_3\}$ of $\H$ as shown in Fig. \ref{S}. Then formula (\ref{BCHI}) implies the result.
	 
	Now let us prove another implication: if $(V_1, V_2, V_3)$ is an orthogonal splitting of $\H$ satisfying \ref{cond}, then this splitting is symmetric. Consider a separating pair $(\gamma, \delta)$ (non necessary symmetric) such that $(V_1, V_2, V_3)$ corresponds to $(\gamma, \delta)$. However, there exist another hyperelliptic involution $s'$ such that $\gamma$ and $\delta$ are $s'$-invariant, see Fig. \ref{S2}.
	
	\begin{figure}[h]
		\begin{center}
			\scalebox{0.5}{
				\begin{tikzpicture}
				\draw[red, very thick, dashed] (2,3) to[out = -100, in = 100] (2, -3);
				\draw[red, very thick, dashed] (-2,3) to[out = -100, in = 100] (-2, -3);
				
				\draw[very thick] (-4, 3) to (4, 3);
				\draw[very thick] (-4, -3) to (4, -3);
				\draw[very thick] (-4, 3) arc (90:270:3);
				\draw[very thick] (4, 3) arc (90:-90:3);
				\draw[very thick] (0, 0) circle (1);
				\draw[very thick] (-4, 0) circle (1);
				\draw[very thick] (4, 0) circle (1);
				
				\draw[blue, very thick] (4,3) to[out = -70, in = 70] (4, 1);
				\draw[blue, very thick, dashed] (4,3) to[out = -110, in = 110] (4, 1);
				
				\draw[blue, very thick] (0,3) to[out = -70, in = 70] (0, 1);
				\draw[blue, very thick, dashed] (0,3) to[out = -110, in = 110] (0, 1);
				
				\draw[blue, very thick] (-4,3) to[out = -70, in = 70] (-4, 1);
				\draw[blue, very thick, dashed] (-4,3) to[out = -110, in = 110] (-4, 1);
				
				\draw[red, very thick] (2,3) to[out = -80, in = 80] (2, -3);
				\draw[red, very thick] (-2,3) to[out = -80, in = 80] (-2, -3);
				
				\node[red, scale = 2] at (2.7, 2) {$\delta$};
				\node[red, scale = 2] at (2.7-4, 2) {$\gamma$};
				
				\node[blue, scale = 2] at (4.75, 2) {$\alpha'_3$};
				\node[blue, scale = 2] at (4.75-4, 2) {$\alpha'_2$};
				\node[blue, scale = 2] at (4.75-8, 2) {$\alpha'_1$};
				
				\node[scale = 2] at (4.35, 0) {$\beta'_3$};
				\node[scale = 2] at (4.35-4, 0) {$\beta'_2$};
				\node[scale = 2] at (4.35-8, 0) {$\beta'_1$};

				\draw[very thick, dashed] (8, 0) to (10, 0);
				\draw[->, thick] (10, 2) to[out = -10, in = 10] (10, -2);
				\node[scale = 3] at (12, 0) {$s'$};

				\end{tikzpicture}} \end{center}
		\caption{}
		\label{S2}
	\end{figure}
	Denote $a_i' = [\alpha_i']$ and $b_i' = [\beta_i']$.
	Since $\Arf(\omega_0|_{V_1}) = \Arf(\omega_0|_{V_3}) = 0$, then
	\begin{equation*} 
	\omega_0(a_1') = \omega_0(b_1') = \omega_0(a_3') = \omega_0(b_3') = 1.
	\end{equation*}
	Also, since $\Arf(\omega_0) = 0$, the following cases are possible.
	
	(1) $\omega_0(a_2') = 0$, $\omega_0(b_2') = 0$.
	
	(2) $\omega_0(a_2') = 0$, $\omega_0(b_2') = 1$.
	
	(3) $\omega_0(a_2') = 1$, $\omega_0(b_2') = 0$.\\
	We define an element $\phi \in \Mod(\S_3)$ as follows.
	
	In the case (1) we define $\phi = T_{\alpha_2'}$.
	
	In the case (2) we define $\phi = \id$.
	
	In the case (3) we define $\phi =  T_{\beta_2'}  T_{\alpha_2'}$.\\
	Consider the involution $s'' = \phi^{-1} \circ s' \circ \phi$. In order to study it we consider in Fig. \ref{S3} the embedding of $\S_3$ that differs from the one shown in Fig. \ref{S2} by a representative of $\phi^{-1}$.
	
		\begin{figure}[h]
		\begin{center}
			\scalebox{0.5}{
				\begin{tikzpicture}
				\draw[red, very thick, dashed] (2,3) to[out = -100, in = 100] (2, -3);
				\draw[red, very thick, dashed] (-2,3) to[out = -100, in = 100] (-2, -3);
				
				\draw[very thick] (-4, 3) to (4, 3);
				\draw[very thick] (-4, -3) to (4, -3);
				\draw[very thick] (-4, 3) arc (90:270:3);
				\draw[very thick] (4, 3) arc (90:-90:3);
				\draw[very thick] (0, 0) circle (1);
				\draw[very thick] (-4, 0) circle (1);
				\draw[very thick] (4, 0) circle (1);
				
				\draw[blue, very thick] (4,3) to[out = -70, in = 70] (4, 1);
				\draw[blue, very thick, dashed] (4,3) to[out = -110, in = 110] (4, 1);
				
				\draw[blue, very thick] (0,3) to[out = -70, in = 70] (0, 1);
				\draw[blue, very thick, dashed] (0,3) to[out = -110, in = 110] (0, 1);
				
				\draw[blue, very thick] (-4,3) to[out = -70, in = 70] (-4, 1);
				\draw[blue, very thick, dashed] (-4,3) to[out = -110, in = 110] (-4, 1);
				
				\draw[red, very thick] (2,3) to[out = -80, in = 80] (2, -3);
				\draw[red, very thick] (-2,3) to[out = -80, in = 80] (-2, -3);
				
				\node[red, scale = 2] at (2.7, 2) {$\delta$};
				\node[red, scale = 2] at (2.7-4, 2) {$\gamma$};
				
				\node[blue, scale = 2] at (4.75, 2) {$\alpha'_3$};
				\node[blue, scale = 2] at (4.75-4, 2) {$\alpha''_2$};
				\node[blue, scale = 2] at (4.75-8, 2) {$\alpha'_1$};
				
				\node[scale = 2] at (4.35, 0) {$\beta'_3$};
				\node[scale = 2] at (4.35-4, 0) {$\beta''_2$};
				\node[scale = 2] at (4.35-8, 0) {$\beta'_1$};

				\draw[very thick, dashed] (8, 0) to (10, 0);
				\draw[->, thick] (10, 2) to[out = -10, in = 10] (10, -2);
				\node[scale = 3] at (12, 0) {$s''$};

				\end{tikzpicture}} \end{center}
		\caption{}
		\label{S3}
	\end{figure}
	We have $\alpha_2'' = \phi(\alpha_2')$ and $\beta_2'' = \phi(\beta_2')$. Denote $a_2'' = [\alpha_2'']$ and $b_2'' = [\beta_2'']$. Straightforward computation shows that these homology classes are as follows.
	
	In the case (1) we have $a_2'' = a_2'$ and $b_2'' = b_2' + a_2'$.
	
	In the case (2) we have $a_2'' = a_2'$ and $b_2'' = b_2'$.
	
	In the case (3) we have $a_2'' = a_2 + b_2'$ and $b_2'' = 2 b_2' + a_2'$.\\
	Straightforward computations show that $\omega_0(a_2'') = 0$ and $\omega_0(b_2'') = 1$. Formula (\ref{BCHI}) implies that $\rho_{\omega_0}(s'') = 1$. Hence Corollary \ref{fact} imply that there exists $h \in \SI_3$ such that $h \circ s'' \circ h^{-1} = s$. Then the curves $h(\gamma)$ and $h(\delta)$ are $s$-invariant, so the separating pair $(h(\gamma), h(\delta))$ is symmetric. Since $h \in \SI_3$, then the splitting $(V_1, V_2, V_3)$ corresponds to $(h(\gamma), h(\delta))$. Therefore $(V_1, V_2, V_3)$ is symmetric.	
\end{proof}

\section{Proof of the linear independence}\label{Sec4}

The main goal of this section is to prove Theorem \ref{mainth}. First we need to recall the construction of the symmetric complex of cycles, introduced be Brendle, Childers and Margalit in \cite{Brendle}.

\subsection{Morita homomorphisms}

First prove that each simple abelian cycle $\A_{V_1, V_2, V_3} \in \H_2(\SI_3, \Z)$ is nonzero and has the infinite order. We need to recall the definition of the Morita homomorphisms.
Recall that the \textit{Johnson kernel} $\K_g$ is the subgroup of $\Mod(\S_g)$ generated by all Dehn twists about separating curves. Since the group $\SI_g$ is generated by Dehn twists about $s$-invariant separating curves \cite{BrendleMP}, then it is contained in $\K_g$. 
	
	Recall the construction of the Morita homomorphisms $\K_g \to \Z$ (see \cite{Morita0, Morita}).
	For an oriented homology $3$-sphere $N$ denote by $\lambda(N) \in \Z$ the Casson invariant of $N$ (see \cite{Sav}). Let $\sigma: \S_g \hookrightarrow N$ be a Heegaard embedding. Denote by $V_+$ and $V_-$ the connected components of $N \setminus f(\S_g)$, such that the orientations of $f(\S_g) = \partial V_+$ induced by the orientations of $\S_g$ and $V_+$ coincide. Then we have the he Seifert bilinear form
	$$l_\sigma: \H_1(\S_g, \Z) \times \H_1(\S_g, \Z) \to \Z$$
	defined by
	$$l_\sigma (x, y) = \lk(\sigma_*(x), \sigma_*(y)^+),$$
	where $\sigma_*(y)^+$ is obtained from $\sigma_*(y)$ by pushing it inside $V_-$ and $\lk(\cdot, \cdot)$ is the linking number. We have
	$$l_\sigma (y, x) = l_\sigma (x, y) + x \cdot y.$$
	For $h \in \K_g$ denote by $N_h$ the homology $3$-sphere obtained from gluing $V_+$ and $V_-$ along $\sigma(\S_g)$ via the map $\sigma^{-1}\circ h \circ \sigma$. Morita \cite[Theorem 2.2]{Morita} proved, that the map $\lambda_\sigma: \K_g \to \Z$ given by  $\lambda_\sigma(h) = \lambda(N_h) - \lambda(N)$ is a well-defined group homomorphism. Moreover, if $\gamma$ is a separating curve on $\S_g$, then $\lambda_\sigma(T_\gamma)$ can be computed explicitly.
	We give this formula in the case $g=3$. Let $\{a, b\}$ be a symplectic basis of the first homology group of the genus-one component bounded by $\gamma$. Then
	\begin{equation}\label{mhom}
	\lambda_\sigma(T_\gamma) = (l_\sigma(a, a)l_\sigma(b, b) - l_\sigma(a, b)l_\sigma(b, a)).
	\end{equation}

	\begin{prop} \label{nonzero}
		Each simple abelian cycle $\A_{V_1, V_2, V_3} \in \H_2(\SI_3, \Z)$ is nonzero and has the infinite order.
	\end{prop}
	
	\begin{proof}
	We will show that for any two nonisotopic disjoint separating curves $\gamma$ and $\delta$ on $\S_3$ the abelian cycle $\A(T_\gamma, T_\delta) \in \H_2(\K_3, \Z)$ is nonzero and has the infinite order. We will use Morita homomorphism and follow ideas of \cite{BrendleFarb} and \cite{Gaifullin_J}.

	Let $\sigma: \S_3 \hookrightarrow S^3$ be the standard embedding and let $\{a_1, b_1, a_2, b_2, a_3, b_3\}$ be a symplectic basis of $\H$ shown in Fig. \ref{S} (in this case $V_+$ is inside $\S_3$). Then we have
	$$l_\sigma (a_i, a_j) = l_\sigma (b_i, b_j) = l_\sigma (a_i, b_j) = 0 \;\; \mbox{for all} \;\; 1 \leq i, j \leq 3,$$
	$$l_\sigma (b_i, a_i) = 1 \;\; \mbox{for all} \;\; 1 \leq i \leq 3,$$
	$$l_\sigma (b_i, a_j) = 0 \;\;\ \mbox{for all} \;\; i \neq j \;\; \mbox{and} \;\; 1 \leq i, j \leq 3.$$ 
	The following idea is taken from the paper \cite{Gaifullin_J} of Gaifullin. Consider the elements $\phi_1, \phi_3 \in \Mod(\S_3)$, such that their images $(\psi_1)_*$ and $(\psi_3)_*$ in $\Sp(6, \Z)$ are given by the following matrices in the standard basis $\{a_1, b_1, a_2, b_2, a_3, b_3\}$.
	$$(\psi_1)_*=
	\begin{pmatrix}
	1 & 1 & 1 & 0 & 0 & 0 \\
	0 & 1 & 1 & 0 & 0 & 0 \\
	1 & 0 & 1 & 1 & 0 & 0 \\
	1 & 0 & 0 & 1 & 0 & 0 \\
	0 & 0 & 0 & 0 & 1 & 0 \\
	0 & 0 & 0 & 0 & 0 & 1 
	\end{pmatrix}, \;\;\; (\psi_3)_*=
	\begin{pmatrix} 
	1 & 0 & 0 & 0 & 0 & 0 \\
	0 & 1 & 0 & 0 & 0 & 0 \\
	0 & 0 & 1 & 1 & 1 & 0 \\
	0 & 0 & 0 & 1 & 1 & 0 \\
	0 & 0 & 1 & 0 & 1 & 1 \\
	0 & 0 & 1 & 0 & 0 & 1 
	\end{pmatrix}.
	$$
	For any element $\phi \in \Mod(\S_3)$ we have $l_{\sigma \circ \phi}(x, y) = l_\sigma(\phi_*(x), \phi_*(y))$. Hence by (\ref{mhom}) we obtain
	$$\lambda_{\sigma \circ \phi_1}(T_\gamma) = (l_\sigma(a_1+a_2+b_2, a_1+a_2+b_2)l_\sigma(a_1+b_1, a_1+b_1) - $$
	$$- l_\sigma(a_1+a_2+b_2, a_1+b_1)l_\sigma(a_1+b_1, a_1+a_2+b_2)) = 1, $$
	$$\lambda_{\sigma \circ \phi_1}(T_\delta) = (l_\sigma(a_3, a_3)l_\sigma(b_3, b_3) - l_\sigma(a_3, b_3)l_\sigma(b_3, a_3)) = 0,$$
	
	$$\lambda_{\sigma \circ \phi_3}(T_\gamma) = (l_\sigma(a_1, a_1)l_\sigma(b_1, b_1) - l_\sigma(a_1, b_1)l_\sigma(b_1, a_1)) = 0,$$
	$$\lambda_{\sigma \circ \phi_3}(T_\delta) = (l_\sigma(a_2+b_2+a_3, a_2+b_2+a_3)l_\sigma(a_3+b_3, a_3+b_3) - $$
	$$- l_\sigma(a_2+b_2+a_3, a_3+b_3)l_\sigma(a_3+b_3, a_2+b_2+a_3)) = 1.$$
	
	Consider the homology class $\Theta = \lambda_{\sigma \circ \phi_1} \cdot \lambda_{\sigma \circ \phi_3} \in \H^2(\K_3, \Z)$. We obtain
	$$\left\langle  \Theta, \A(T_\gamma, T_\delta) \right\rangle = -\det \begin{pmatrix}
	\lambda_{\sigma \circ \phi_1}(T_\gamma) & \lambda_{\sigma \circ \phi_1}(T_\delta) \\
	\lambda_{\sigma \circ \phi_3}(T_\gamma) & \lambda_{\sigma \circ \phi_3}(T_\delta)
	\end{pmatrix} = -\det \begin{pmatrix}
	1 & 0 \\
	0 & 1
	\end{pmatrix} = -1.$$
	This implies that the abelian cycle $\A(T_\gamma, T_\delta) \in \H_2(\K_3, \Z)$ is nonzero and has the infinite order. Since $\SI_3$ is a subgroup of $\K_3$, this implies the Proposition \ref{nonzero}.
\end{proof}

\subsection{Complex of symmetric cycles}
Let us recall the construction of the \textit{complex of symmetric cycles} $\SB_g(x)$ introduced by Brendle, Childers and Margalit in \cite{Brendle}. For an oriented curve $c$ on $\S_g$ we denote by $\overleftarrow{c}$ the reverse curve of $c$. Let us denote by $\C$ the set of all isotopy classes of oriented non-separating simple closed curves on $\S_{g}$.

A \textit{skew-symmetric pair of curves} is a pair of disjoint oriented closed curves $\{c_1, c_2\}$ on $\S_g$ such that $s(c_1) = \overleftarrow{c_2}$ and $s(c_2) = \overleftarrow{c_1}$. Note that $c_1$ and $c_2$ can be isotopic, but cannot be homologicaly trivial. 
A \textit{skew-symmetric multicurve} is a nonempty collection of skew-symmetric pairs of curves that are pairwise disjoint and pairwise non-isotopic.

Let us fix some nonzero homology class $x \in \H_1(\S_g, \Z)$.
A \textit{1-cycle} for $x$ is a finite formal sum
$$c = \sum_{i=1}^n k_i c_i,$$
where $k_i \geq 0$, $c_i \in \C$ and $\sum_{i=1}^n k_i [c_i] = x$. The set $\{c_i\}$ is called the \textit{support} of $c$. 
A \textit{basis skew-symmetric 1-cycle} for $x$ is a 1-cycle
$$\sum_{i=1}^n \dfrac{k_i}{2} (c_i + s(\overleftarrow{c_i})),$$
where the support $\{c_i, s(\overleftarrow{c_i})\}$ is a skew-symmetric multicurve, and where the homology classes $[c_i]$ are linearly independent.

Denote by $\mathcal{SM}$ the set of isotopy classes of skew-symmetric multicurves that are unions of supports of basis skew-symmetric cycles for $x$. Let $M = \{c_1, s(\overleftarrow{c_1}), \dots, c_m, s(\overleftarrow{c_m})\} \in \C$ be a skew-symmetric multicurve. Consider a convex polytope
$$P_M = \{ (k_1, \dots, k_m) \in \R^m \subset \R^\C \; | \; \sum_{i=1}^m \dfrac{k_i}{2} (c_i + s(\overleftarrow{c_i})) \; \mbox{is a skew-symmetric cycle for} \; x\}.$$
The complex of symmetric cycles is defined by $$\SB_g(x) = \bigcup_{M \in \mathcal{SM}} P_M \subset \R^\C.$$
Denote by $\mathcal{SM}_0 \subset \mathcal{SM}$ the subset of skew-symmetric multicurves corresponding to 0-cells of $\SB_g(x).$

\begin{theorem} \cite[Theorem 2.1]{Brendle} \label{Contr}
	Let $g \geq 1$ and let $x \in \H_1(\S_g, \Z)$ be any nonzero primitive element. Then $\SB_g(x)$ is contractible.
\end{theorem}

\subsection{The spectral sequence for a group action on a CW-complex}
Suppose that a group $G$ acts cellularly without rotations on a contractible $CW$-complex $X$. Then there is a corresponding spectral sequence (see (7.7) in \cite[Section VII.7]{Brown}) has the form
\begin{equation} \label{spec_sec}
E_{p, q}^1 \cong \bigoplus_{\sigma \in \X_p}\H_q(\Stab_G (\sigma)) \Rightarrow \H_{p+q}(G, \Z),
\end{equation}
where $\X_p$ is a set containing one representative from each $G$-orbit of $p$-cells of $X$. 

Now let $E_{*,*}^*$ be the spectral sequence (\ref{spec_sec}) for the action on $\SI_g$ on $\SB_g(x)$ for some primitive element $0 \neq x \in \H_{1}(\S_g, \Z)$.
The fact that $\SI_g$ acts on $\SB_g(x)$ without rotations follows from the result of Ivanov \cite[Theorem 1.2]{Ivanov}.
Brendle, Childers and Margalit proved \cite[Proposition 4.12]{Brendle} that for each cell $\sigma \in \SB_g(x)$ we have
\begin{equation} \label{ineq}
\dim(\sigma) + \cd(\Stab_{\SI_g}(\sigma)) \leq g-1.
\end{equation}
This immediately implies $E^1_{p, q} = 0$ for $p+q > g-1$. Hence all differentials $d^1, d^2, \dots$ to the group $E^1_{0, g-1}$ are trivial, so $E^1_{0, g-1} = E^{\infty}_{0, g-1}$. Therefore we have the following result.
\begin{prop} \label{prop1}
	Let $\mathfrak{M} \subseteq \mathcal{SM}_{0}(x)$ be a subset consisting of skew-symmetric multicurves from pairwise different $\SI_g$-orbits. Then the inclusions $\Stab_{\SI_g}(M) \subseteq \SI_g$, where $M \in \mathfrak{M}$, induce an injective homomorphism
	\begin{equation*} \label{c-l}
	\bigoplus_{M \in \mathfrak{M}} \H_{g-1}(\Stab_{\SI_g}(M), \Z) \hookrightarrow \H_{g-1}(\SI_g), \Z).
	\end{equation*}
\end{prop}

\subsection{Symmetric curves on one- and two-punctured tori}

Consider a torus $\mathbb{T}$ with one marked point $p$, equipped with a hyperelliptic involution $s$ stabilising $p$. We will conveniently consider marked points instead of punctures, because we will work in terms of the homology groups of the torus without punctures.

\begin{lemma} \label{t1}
	Any nonzero primitive homology class $w \in \H_1(\mathbb{T}, \Z)$ can be realised by an $s$-invariant simple closed curve on $\mathbb{T}$ disjoint from $p$
\end{lemma}
\begin{proof}
	Let $\mathbb{T} = \R^2 / \Gamma$, where $\Gamma = 2\Z \oplus 2\Z \hookrightarrow \R \oplus \R = \R^2$. It is convenient to consider the involution $s: (x, y) \mapsto (2-x, 2-y) \mod \Gamma$ and the marked point $p = (1, 1) \mod \Gamma$. Any primitive homology class can be realized by the curve $ax+by+c=0 \mod 2$ with $a, b \in \Z, c \in \R$ and $\gcd(a, b) = 1$. We claim that this curve is symmetric if and only if $c=0 \mod 2$ or $c=1 \mod 2$. Indeed,
	$$a(2-x) + b(2-y) + c = -ax - by - c + 2a + 2b +2c = -(ax+by+c) + 2(a+b) + 2c = 2c \mod 2,$$
	because $2(a+b)=0 \mod 2$. Hence we need $2c=0 \mod 2$, which implies $c=0 \mod 2$ or $c=1 \mod 2$.
	
	If $a+b = 1 \mod 2$, then we take $c=0$. Let us check that in this case the curve $ax+by=0$ is disjoint from $p$. Indeed, if $(1, 1)$ satisfy this equation, then $a+b = 0 \mod 2$, which is a contradiction.
	
	If $a+b = 0 \mod 2$, then we take $c=1$. Let us check that in this case the curve $ax+by=1$ is disjoint from $p$. Indeed, if $(1, 1)$ satisfy this equation, then $a+b+1=1 \mod 2$, which is also contradiction.
\end{proof}

Now consider a torus $\mathbb{T}$ with two marked point $p$ and $q$, equipped with a hyperelliptic involution $s$ stabilising each of the points $p$ and $q$. Choose a symplectic basis $\{e_1, e_2\}$ of $\H_1(\mathbb{T}, \Z)$, such that $e_2$ can be realized on $\mathbb{T}$ by an $s$-invariant simple closed curve disjoin from $p$ and $q$. Define an Sp-quadratic form $\nu$ on $\H_1(\mathbb{T}, \Z)$ by $\nu(e_1) = 0$ and $\nu(e_2) = 1$.
\begin{lemma} \label{t21}
	Any nonzero primitive homology class $w \in \H_1(\mathbb{T}, \Z)$ with $\nu(y) = 1$ can be realised by an $s$-invariant simple closed curve on $\mathbb{T}$ disjoint from $p$ and $q$.
\end{lemma}
\begin{proof}
	As in the proof of Lemma \ref{t1}, let $\mathbb{T} = \R^2 / \Gamma$, where $\Gamma = 2\Z \oplus 2\Z \hookrightarrow \R \oplus \R = \R^2$. The involution is  $s: (x, y) \mapsto (2-x, 2-y) \mod \Gamma$. Now we take $p=(1, 1)$ and $q=(1, 0)$. We can assume that $e_1 = (2, 0)$ and $e_2=(0, 2)$. One can easily check that a curve $ax+by+c=0$ defines a homology class $w$ with $\nu(w) = 1$ if and only if $a = 1 \mod 2$ and $b=0 \mod 2$.
	
	We take $c=0$. It is suffices to check that in this case the curve $ax+by=0$ defining a homology class $w$ with $\nu(w) = 1$ is disjoint from $p$ and $q$ (we have already showed in the proof on Lemma \ref{t1} that such a curve is symmetric). Indeed, after substitution $p=(1, 1)$ and $q=(1, 0)$ in the equation $ax+by=0$, we obtain
	$a+b = 0 \mod 2$
	and
	$a = 0 \mod 2.$
	Both of these conditions do not hold because we have $a = 1 \mod 2$ and $b=0 \mod 2$.
\end{proof}
\begin{lemma} \label{t22}
	Any nonzero primitive homology class $w \in \H_1(\mathbb{T}, \Z)$ with $\nu(w) = 0$ can be realised by a simple closed curve $\zeta$ on $\mathbb{T}$ disjoint from $p$ and $q$, such that $\zeta \cap s(\zeta) = \varnothing$ and such that $\zeta$ and $s(\zeta)$ are not isotopic on $\mathbb{T} \setminus \{p, q\}$.
\end{lemma}
\begin{proof}
	Let $\mathbb{T}$, the involution $s$, the points $p, q$ and the homology classes $e_1, e_2$ be as in the proof of Lemma \ref{t22}. Consider the curve $\zeta$ representing a primitive homology class $w$ with $\nu(w) = 0$, defined by the equation $ax+by+\frac{1}{2}=0$. Since $s(\zeta)$ is given by the equation $ax+by+\frac{3}{2}=0$, then $\zeta \cap s(\zeta) = \varnothing$. Moreover, $\zeta$ is disjoint from $p$ and $q$. It suffices to check that $\zeta$ and $s(\zeta)$ are not isotopic.
	
	Since $w$ is a primitive element, then there exists $z \in \H_1(\mathbb{T}, \Z)$ such that $(w, z)$ is a symplectic basis of $\H_1(\mathbb{T}, \Z)$. 
	Let $A \in \SL(2, \Z)$ be a matrix such that $Ae_1 = w$ and $Ae_2 = z$. Since $\nu(w) = 0$, then $a = 0 \mod 2$ or $b=1 \mod 2$. Hence $A \mod 2$ has one of the following forms:
	\begin{equation} \label{Amat}
	\begin{pmatrix}
	1 & 0 \\
	0 & 1 
	\end{pmatrix}, \begin{pmatrix}
	1 & 1 \\
	0 & 1 
	\end{pmatrix}, \begin{pmatrix}
	1 & 0 \\
	1 & 1 
	\end{pmatrix}.
	\end{equation}
	
	Consider the linear automorphism of $\mathbb{T}$ given by the matrix $A^{-1}$. The curves $A^{-1}(\zeta)$ and $A^{-1}(s(\zeta))$ are given by the equations $y = \frac{1}{2}$ and $y = \frac{3}{2}$, respectively. Since $A = A^{-1} \mod 2$, we have that the set $\{A^{-1}(p), A^{-1}(q)\}$ coincides with one of the following sets corresponding to the matrices (\ref{Amat}):
	$$\{(1, 1), (1, 0)\}, \;\; \{(0, 1), (1, 0)\}, \;\; \{(1, 0), (1, 1)\}.$$
	In each of these cases the cures $A^{-1}(\zeta)$ and $A^{-1}(s(\zeta))$ are not isotopic on $\mathbb{T} \setminus \{A^{-1}(p), A^{-1}(q)\}$. Hence we obtain that $\zeta$ and $s(\zeta)$ are also not isotopic on $\mathbb{T} \setminus \{p, q\}$.
\end{proof}

\subsection{Proof of Theorem \ref{mainth}}

\begin{proof}[Proof of Theorem \ref{mainth}]
	Assume the converse and consider a nontrivial relation
	\begin{equation}\label{rel}
	\sum_{i=1}^n \lambda_i \A_{V^i_1, V^i_2, V^i_3} = 0, \;\; \lambda_i \in \Z,
	\end{equation}
	where $\{(V^i_1, V^i_2, V^i_3) \; | \; i = 1, \dots, n\}$ is a family of orthogonal splittings of $\H$, that are pairwise different as the unordered splittings ans satisfy the condition (\ref{cond}).
	For each element $x \in \H$ there exists unique decompositions
	$$x = x^i_1+x^i_2 + x^i_3,  \;\;\;\; \; x^i_j \in V^i_j, \; \; i = 1, \dots, n, \;\; j = 1, 2, 3.$$ The following result is proved in \cite{Gaifullin_J} for arbitrary genus.
	\begin{lemma} \cite[Lemma 4.5]{Gaifullin_J}\label{ex}
		There exists a nonzero primitive homology class $x \in \H$ such that
		
		(1) all homology classes $x^i_j$ are nonzero, $i = 1, \dots, n, \;\; j = 1, 2, 3,$
		
		(2) if $1 \leq s < t \leq n$, then the unordered sets $\{x^s_1, x^s_2, x^s_3\}$ and $\{x^t_1, x^t_2, x^t_3\}$ do not coincide.
	\end{lemma}

Choose a primitive homology class $x \in \H$ satisfying the conditions (1) and (2) of Lemma \ref{ex}. Then every $x^i_j$ can be uniquely written as $x^i_j = n^i_j a^i_j$, where $n^i_j \in \mathbb{N}$ and $a^i_j \in \H$ is a primitive homology class. This following statement is also proved in \cite{Gaifullin_J} for arbitrary genus.
\begin{lemma}  \cite[after proof of Lemma 4.5]{Gaifullin_J} \label{ex2}
	If $1 \leq s < t \leq n$, then the unordered sets $\{a^s_1, a^s_2, a^s_3\}$ and $\{a^t_1, a^t_2, a^t_3\}$ do not coincide.
\end{lemma}

For each $i$ let $(\gamma^i, \delta^i)$ be asymmetric separating pair such that the splitting $(V^i_1, V^i_2, V^i_3)$ corresponds to $(\gamma^i, \delta^i)$. Let $X^i_1, X^i_2, X^i_3$ be the connsected components of $\S_3 \setminus \{\gamma^i, \delta^i\}$ such that $V^i_j = \H_1(X^i_j, \Z)$ for $j = 1, 2, 3$. Since $X^i_1$ is a one-punctured torus, then by Lemma \ref{t1} the homology class $a^i_1$ can be realised by a curve $\eta^i_1$ on $X^i_1$ disjoint from $\gamma^i$. Similarly, the homology class $a^i_3$ can be realised by a curve $\eta^i_3$ on $X^i_3$ disjoint from $\delta^i$.

For the homology class $a^i_2$ we have two possibilities: either $\omega_0(a^i_2) = 1$ or $\omega_0(a^i_2) = 0$. If $\omega_0(a^i_2) = 1$ then by Lemma \ref{t21} $a^i_2$  can be realised by an $s$-invariant simple closed curve $\eta^i_2$ on $X^i_2$ disjoint from $\gamma^i$ and $\delta^i$. If $\omega_0(a^i_2) = o$ then by Lemma \ref{t22} $a^i_2$  can be realised by a simple closed curve $\eta^i_2$ on $\mathbb{T}$ disjoint from $p$ and $q$, such that $\eta^i_2 \cap s(\nu^i_2) = \varnothing$ and such that $\eta^i_2$ and $s(\eta^i_2)$ are not isotopic.

For each $i$ consider the basis skew-symmetric 1-cycle for $x$ defined by
$$C_i = \sum_{i=j}^3 \dfrac{n_j}{2} \Bigl(\nu^i_j + s\bigl(\overleftarrow{\nu^i_j}\bigr)\Bigr).$$
Denote by $M^i$ its support. We will prove the following statement in the next section.

Now let us finish the proof of Theorem \ref{mainth}. By Proposition \ref{prop1} the inclusions $\Stab_{\SI_3}(M^i) \hookrightarrow \SI_3$ induce an injective homomorphism
\begin{equation} \label{c-l2}
\iota: \bigoplus_{i=1}^n \H_{2}(\Stab_{\SI_3}(M^i), \Z) \hookrightarrow \H_{2}(\SI_3), \Z).
\end{equation}
Since $\iota\Bigl(\A(T_{\gamma^i}, T_{\delta^i})\Bigr) = \A_{V^i_1, V^i_2, V^i_3}$, then our assumption (\ref{rel}) implies 
$$\iota(\sum_{i=1}^n \lambda \A(T_{\gamma^i}, T_{\delta^i})) = 0.$$ 
By Proposition \ref{nonzero} we obtain that for each $i$ the homology class $$\A(T_{\gamma^i}, T_{\delta^i}) \in \H_{2}(\Stab_{\SI_3}(M^i), \Z)$$ 
is nonzero and has the infinite order, so $\lambda_i = 0$ for all $i = 1, \dots, n$. This concludes the proof of Theorem \ref{mainth}.
\end{proof}

\end{document}